\documentclass[A4, 11pt]{article}
\usepackage{amscd, amssymb, amsmath, amsthm}
\usepackage{latexsym}
\usepackage{upref}
\usepackage{epsfig}
\usepackage{mathrsfs}
\usepackage{float, floatflt}
\usepackage{indentfirst}
\usepackage{hyperref}
\hypersetup{
    colorlinks=true,
    linkcolor=blue,
    citecolor=blue,
    filecolor=blue,
    urlcolor=blue,
}
\usepackage{graphics,graphicx,color}
\usepackage{cite}
\usepackage{enumitem}
\usepackage{fancyhdr} 
\usepackage[T1]{fontenc}
\usepackage{mathptmx} 

\AtBeginDocument{
  \DeclareSymbolFont{AMSb}{U}{msb}{m}{n}
  \DeclareSymbolFontAlphabet{\mathbb}{AMSb}}  
  
\usepackage{rotating} 
\usepackage{booktabs} 
\usepackage{pb-diagram,pb-xy}       
\usepackage{tikz}                   
\usetikzlibrary{matrix,arrows}
\usepackage{multirow} 
\usepackage{color, colortbl} 
\definecolor{Gray}{gray}{0.93} 
\definecolor{LightCyan}{rgb}{0.88,1,1}

\theoremstyle{plain}
\newtheorem{theorem}{Theorem}[section]

\newtheorem{corollary}[theorem]{Corollary}
\theoremstyle{definition}

\theoremstyle{remark}

\numberwithin{equation}{section}

\makeatletter
\def\th@plain{%
  \thm@notefont{}
  \itshape 
}
\def\th@definition{%
  \thm@notefont{}
  \normalfont 
} \makeatother

\setlist{font=\normalfont}

\DeclareMathAlphabet{\cols}{OMS}{cmsy}{m}{n} %

\newcommand{\R}{\mathbb{R}}
\newcommand{\B}{\mathbb{B}}
\newcommand{\set}[1]{\{#1\}}
\newcommand{\cset}[2]{\set{{#1}\colon{#2}}}
\newcommand{\norm}[1]{\|#1\|}
\newcommand{\Bp}[1]{\left(#1\right)}
\newcommand{\gen}[1]{\langle#1\rangle}
\DeclareSymbolFont{newfont}{OML}{cmm}{m}{it} 
\DeclareMathSymbol{\Varrho}{3}{newfont}{37}
\newcommand{\gyr}[2]{{\mathrm{gyr}[{#1}]}{#2}}
\newcommand{\aut}[1]{\mathrm{Aut}\,{(#1)}}
\newcommand{\igyr}[2]{{\mathrm{gyr^{-1}}[{#1}]}{#2}}

\newcommand{\Or}[1]{\mathrm{O}({#1})}
\newcommand{\res}[2]{{#1}\hskip-3pt\mid_{#2}}
\newcommand{\Gyr}[2]{{\mathrm{Gyr}[{#1}]}{#2}}
\newcommand{\Iso}[1]{\mathrm{Iso}\,{(#1)}}

\newcommand{\vrho}{\Varrho{\hskip-1.5pt}}

\renewcommand{\vec}[1]{\mathbf{#1}}

\begin{document}
\title{\bf{The isometry group of $n$-dimensional Einstein gyrogroup}}
\author{
Teerapong Suksumran\,\href{https://orcid.org/0000-0002-1239-5586}{\includegraphics[scale=1]{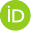}}\\
Department of Mathematics\\
Faculty of Science, Chiang Mai University\\
Chiang Mai 50200, Thailand\\
{\tt teerapong.suksumran@cmu.ac.th}}
\date{}
\maketitle


\begin{abstract}
The space of $n$-dimensional relativistic velocities normalized to $c = 1$, $$\mathbb{B} = \{\vec{v}\in\mathbb{R}^n\colon \|\vec{v}\| < 1\},$$ 
is naturally associated with Einstein velocity addition $\oplus_E$, which induces the \mbox{rapidity} metric $d_E$ on $\B$ given by $d_E(\vec{u}, \vec{v}) = \tanh^{-1}\|-\vec{u}\oplus_E\vec{v}\|$. This metric is also known as the Cayley--Klein metric. We give a complete description of the isometry group of $(\B, d_E)$, along with its composition law.
\end{abstract}
\textbf{Keywords.} Einstein velocity addition, Einstein gyrogroup, Cayley--Klein metric, gyrometric, isometry group.\\[3pt]
\textbf{2010 MSC.} Primary 51A05; Secondary 83A05, 51F25, 20N05.
\thispagestyle{empty}

\section{Introduction}
The space of $n$-dimensional relativistic velocities normalized to $c = 1$, 
$$\B = \cset{\vec{v}\in\R^n}{\norm{\vec{v}} < 1},$$ 
has various underlying mathematical structures, including a bounded symmetric space structure \cite{SKJL2011SBL, MR1290679} and a gyrovector space structure \cite{AU2008AHG}. Further, it is a primary object in special relativity in the case when $n = 3$ \cite{YFTS2005PAH, MR3455332}. Of particular importance is reflected in the composition law of Lorentz boosts:
$$
L(\vec{u})\circ L(\vec{v}) = L(\vec{u}\oplus_E\vec{v})\circ
\Gyr{\vec{u},\vec{v}}{},
$$
where $L(\vec{u})$ and $L(\vec{v})$ are Lorentz boosts parametrized by $\vec{u}$ and $\vec{v}$ respectively, $\oplus_E$ is Einstein velocity addition (defined below), and $\Gyr{\vec{u},\vec{v}}{}$ is a rotation of spacetime coordinates induced by an Einstein addition preserving map (namely an Einstein gyroautomorphism) \cite[p. 448]{AU2008AHG}. Moreover, the unit ball $\B$ gives rise to a model for $n$-dimensional hyperbolic geometry when it is endowed with the Cayley--Klein metric as well as the Poincar\'{e} metric \cite{SKJL2013UBL, JR2006FHM}. 

The open unit ball of $\R^n$ admits a group-like structure when it is endowed with {\it Einstein addition} $\oplus_E$, defined by
\begin{equation}\label{eqn: Euclidean Einstein addition}
\vec{u}\oplus_E\vec{v} =
\dfrac{1}{1+\gen{\vec{u},\vec{v}}}\Bp{\vec{u}+
\dfrac{1}{\gamma_{\vec{u}}}\vec{v} +
\dfrac{\gamma_{\vec{u}}}{1+\gamma_{\vec{u}}}\gen{\vec{u},\vec{v}}\vec{u}},
\end{equation}
where $\gen{\cdot, \cdot}$ denotes the usual Euclidean inner product and $\gamma_{\vec{u}}$ is the {\it Lorentz factor} normalized to $c = 1$ given by $\gamma_{\vec{u}} = \dfrac{1}{\sqrt{1-\norm{\vec{u}}^2}}$. In fact, the space $(\B, \oplus_E)$ satisfies the following properties \cite{TSKW2015EGB, AU2008AHG}:
\begin{enumerate}[label=\Roman*.]
    \item\label{item: identity} ({\sc identity}) The zero vector $\vec{0}$ satisfies $\vec{0}\oplus_E \vec{v} = \vec{v} = \vec{v}\oplus_E\vec{0}$ for all $\vec{v}\in \B$.
    \item ({\sc inverse}) For each $\vec{v}\in \B$, the negative vector $-\vec{v}$ belongs to $\B$ and satisfies $$(-\vec{v})\oplus_E \vec{v} = \vec{0} = \vec{v}\oplus_E(-\vec{v}).$$
    \item\label{item: gyroassociative law} ({\sc the gyroassociative law}) For all $\vec{u}, \vec{v}\in \B$, there are Einstein addition preserving bijective self-maps $\gyr{\vec{u},\vec{v}}{}$ and $\gyr{\vec{v},\vec{u}}{}$ of $\B$ such that
\begin{equation*}
\vec{u}\oplus_E (\vec{v}\oplus_E \vec{w}) = (\vec{u}\oplus_E \vec{v})\oplus_E\gyr{\vec{u}, \vec{v}}{\vec{w}}
\end{equation*}
and 
\begin{equation*}
(\vec{u}\oplus_E \vec{v})\oplus_E \vec{w} = \vec{u}\oplus_E (\vec{v}\oplus_E\gyr{\vec{v}, \vec{u}}{\vec{w}})
\end{equation*}
for all $\vec{w}\in \B$.
    \item\label{item: loop property} ({\sc the loop property}) For all $\vec{u},\vec{v}\in \B$, 
$$
\gyr{\vec{u}\oplus_E \vec{v}, \vec{v}}{} = \gyr{\vec{u}, \vec{v}}{}\quad\textrm{and}\quad \gyr{\vec{u}, \vec{v}\oplus_E \vec{u}}{} = \gyr{\vec{u}, \vec{v}}{}.
$$    
\item\label{item: gyrocommutative} ({\sc the gyrocommutative law}) For all $\vec{u},\vec{v}\in \B$, 
$$
\vec{u}\oplus_E\vec{v} = \gyr{\vec{u}, \vec{v}}{(\vec{v}\oplus_E\vec{u})}.
$$
\end{enumerate}
From properties I through V, it follows that $(\B, \oplus_E)$ forms a {\it gyrocommutative \mbox{gyrogroup}} (also called a {\it K-loop} or {\it Bruck loop}), which shares several properties with groups \cite{TS2016TAG, AU2008AHG}. However, Einstein addition is a nonassociative operation so that $(\B, \oplus_E)$ fails to form a group. Property III resembles the associative law in groups and property V resembles the commutative law in abelian groups. The map $\gyr{\vec{u}, \vec{v}}{}$ in property III is called an {\it Einstein gyroautomorphism}, which turns out to be a \mbox{rotation} of the unit ball. Henceforward, $(\B, \oplus_E)$ is referred to as the ($n$-dimensional) {\it Einstein gyrogroup}.

Recall that the {\it rapidity} of a vector $\vec{v}$ in $\B$ (cf. \cite[p. 1229]{SKJL2013UBL}) is defined by 
\begin{equation}
\phi(\vec{v}) = \tanh^{-1}{\norm{\vec{v}}}.
\end{equation}

\begin{theorem}\label{thm: property of rapidity}
The rapidity $\phi$ satisfies the following properties:
\begin{enumerate}
\item\label{item: positivity} $\phi(\vec{v})\geq 0$ and $\phi(\vec{v}) = 0$ if and only if $\vec{v} = \vec{0}$;
\item\label{item: invariant inverse} $\phi(-\vec{v}) = \phi(\vec{v})$;
\item\label{item: subadditivity} $\phi(\vec{u}\oplus_E \vec{v})\leq \phi(\vec{u}) + \phi(\vec{v})$;
\item\label{item: invariant gyration} $\phi(\gyr{\vec{u}, \vec{v}}{\vec{w}}) = \phi(\vec{w})$
\end{enumerate}
for all $\vec{u}, \vec{v}, \vec{w}\in\B$.
\end{theorem}
\begin{proof}
Items \ref{item: positivity} and \ref{item: invariant inverse} are clear. To prove item \ref{item: subadditivity}, let $\vec{u}, \vec{v}\in\B$. By Proposition 3.3 of \cite{SKJL2013UBL}, $\norm{\vec{u}\oplus_E\vec{v}} \leq \dfrac{\norm{\vec{u}} + \norm{\vec{v}}}{1+\norm{\vec{u}}\norm{\vec{v}}}$. Set $u = \tanh^{-1}\norm{\vec{u}}$ and $v = \tanh^{-1}\norm{\vec{v}}$. Then
$$
\norm{\vec{u}\oplus_E\vec{v}} \leq \dfrac{\tanh{u} + \tanh{v}}{1+(\tanh u)(\tanh v)} = \tanh{(u+v)},
$$
which implies $\phi(\vec{u}\oplus_E \vec{v}) = \tanh^{-1}{\norm{\vec{u}\oplus_E\vec{v}}} \leq u+v = \phi(\vec{u})+\phi(\vec{v})$. Item \ref{item: invariant gyration} \mbox{follows} from the fact that any gyroautomorphism of the Einstein gyrogroup is indeed the restriction of an orthogonal transformation of $\R^n$ to $\B$ so that it preserves the \mbox{Euclidean} norm (and hence also the rapidity); see, for instance, Theorem 3 of \cite{TSKW2015EGB} and Proposition 2.4 of \cite{SKJL2013UBL}.
\end{proof}

Theorem \ref{thm: property of rapidity} implies that $d_E$ defined by
\begin{equation}\label{eqn: rapidity metric on Einstein gyrogroup}
d_E(\vec{u}, \vec{v}) = \phi(-\vec{u}\oplus_E\vec{v}) = \tanh^{-1}\norm{-\vec{u}\oplus_E\vec{v}}
\end{equation}
for all $\vec{u},\vec{v}\in\B$ is indeed a metric (or a distance function) on $\B$, called the {\it \mbox{rapidity} metric} of the Einstein gyrogroup. In Theorem 3.9 of \cite{SKJL2013UBL}, Kim and Lawson prove that $d_E$ agrees with the {\it Cayley--Klein metric}, defined from cross-ratios, in the Beltrami--Klein model of $n$-dimensional hyperbolic geometry. Equation \eqref{eqn: rapidity metric on Einstein gyrogroup} includes what \mbox{Ungar} refers to as the (Einstein) gyrometric, which is defined by
\begin{equation}
\vrho_E(\vec{u}, \vec{v}) = \norm{-\vec{u}\oplus_E\vec{v}}
\end{equation}
for all $\vec{u},\vec{v}\in\B$. Using Proposition 3.3 of \cite{SKJL2013UBL}, we obtain that $$\norm{\vec{u}\oplus_E\vec{v}} \leq  \dfrac{\norm{\vec{u}} + \norm{\vec{v}}}{1+\norm{\vec{u}}\norm{\vec{v}}} \leq \norm{\vec{u}}+\norm{\vec{v}}$$ for all $\vec{u},\vec{v}\in\B$ and so the gyrometric $\vrho_E$ is indeed a metric on $\B$. In fact, this is a consequence of Theorem 3.2 of \cite{TS2018MSN}. Since $\tanh^{-1}$ is an injective function, it follows that a self-map of $\B$ preserves $d_E$ if and only if it preserves $\vrho_E$. Hence, $(\B, d_E)$ and $(\B, \vrho_E)$ have the same isometry group.

The next theorem lists some useful algebraic properties of the Einstein gyro-group, which will be essential in studying the geometric structure of the unit ball in Section \ref{sec: main result}.

\begin{theorem}[See \cite{AU2008AHG, TS2016TAG}]\label{thm: basic propety of B}
The following properties are true in $(\B, \oplus_E)$:
\begin{enumerate}
\item $-\vec{u}\oplus_E (\vec{u}\oplus_E \vec{v}) = \vec{v}$;\hfill{\normalfont\sc (left cancellation law)}
\item $-(\vec{u}\oplus_E \vec{v}) = \gyr{\vec{u}, \vec{v}}{(-\vec{v}\oplus_E -\vec{u})}$;
\item $(-\vec{u}\oplus_E \vec{v})\oplus_E \gyr{-\vec{u}, \vec{v}}{(-\vec{v}\oplus_E \vec{w})} = -\vec{u}\oplus_E \vec{w}$;
\item $\gyr{-\vec{u}, -\vec{v}}{} = \gyr{\vec{u}, \vec{v}}{}$;\hfill{\normalfont\sc (even property)}
\item $\gyr{\vec{v}, \vec{u}}{} = \igyr{\vec{u}, \vec{v}}{}$, where $\igyr{\vec{u}, \vec{v}}{}$ denotes the inverse of $\gyr{\vec{u}, \vec{v}}{}$ with respect to composition of functions;\hfill{\normalfont\sc (inversive symmetry)}
\item\label{item: hyperbolic translation} $L_{\vec{u}}\colon \vec{v}\mapsto \vec{u}\oplus_E \vec{v}$ defines a bijective self-map of $\B$ and $L_{\vec{u}}^{-1} = L_{-\vec{u}}$.
\end{enumerate}
\end{theorem}

\section{Main Results}\label{sec: main result}

Let $\Or{\R^n}$ be the orthogonal group of $n$-dimensional Euclidean space $\R^n$; that is, $\Or{\R^n}$ consists precisely of (bijective) Euclidean inner product preserving transformations of $\R^n$ (also called orthogonal transformations of $\R^n$). Since  the unit ball $\B$ is invariant under orthogonal transformations of $\R^n$, it follows that the set
\begin{equation}
\Or{\B} = \cset{\res{\tau}{\B}}{\tau\in\Or{\R^n}},
\end{equation}
where $\res{\tau}{\B}$ denotes the restriction of $\tau$ to $\B$, forms a group under composition of functions. Note that Einstein addition is defined entirely in terms of vector addition, scalar multiplication, and the Euclidean inner product. Hence, every orthogonal transformation of $\R^n$ restricts to an automorphism of $\B$ that leaves the Euclidean norm invariant. In particular, the map $\iota\colon \vec{v}\mapsto - \vec{v}$ defines an automorphism of $\B$. Note that $\Or{\B}$ is a subgroup of the (algebraic) automorphism group of $(\B, d_E)$, denoted by $\aut{\B, d_E}$.

Let $\vec{u},\vec{v}\in \B$. It is not difficult to check that $\gyr{\vec{u}, \vec{v}}{}$ satisfies the following \mbox{properties}:
\begin{enumerate}
\item $\gyr{\vec{u}, \vec{v}}{\vec{0}} = \vec{0}$;
\item $\gyr{\vec{u}, \vec{v}}{}$ is an automorphism of $(\B,\oplus_E)$;
\item $\gyr{\vec{u}, \vec{v}}{}$ preserves the gyrometric $\vrho_E$.
\end{enumerate}
Hence, by Theorem 3.1 of \cite{TA2014GPM}, there is an orthogonal transformation $\phi$ of $\R^n$ for which $\res{\phi}{\B} = \gyr{\vec{u}, \vec{v}}{}$. This proves the following inclusion:
$$
\cset{\gyr{\vec{u}, \vec{v}}{}}{\vec{u},\vec{v}\in\B}\subseteq \Or{\B}.
$$

\begin{theorem}\label{thm: left gyrotranslation isometry}
For all $\vec{u}\in\B$, the left gyrotranslation $L_{\vec{u}}$ defined by $L_{\vec{u}}(\vec{v}) = \vec{u}\oplus_E\vec{v}$ is an isometry of $\B$ with respect to $d_E$.
\end{theorem}
\begin{proof}
Note that $L_{\vec{u}}$ is a bijective self-map of $\B$ since $L_{-\vec{u}}$ acts as its inverse (see, for instance, Theorem 10 (1) of \cite{TSKW2015ITG}). From Theorem \ref{thm: basic propety of B}, we have by inspection that
\begin{align*}
\norm{- (\vec{u}\oplus_E \vec{x})\oplus_E (\vec{u}\oplus_E \vec{y})} &= \norm{\gyr{\vec{u}, \vec{x}}{(- \vec{x} \oplus_E -\vec{u})}\oplus_E(\vec{u}\oplus_E \vec{y})}\\
{} &= \norm{(-\vec{x} \oplus_E -\vec{u})\oplus_E\gyr{\vec{x}, \vec{u}}{(\vec{u}\oplus_E \vec{y})}}\\
{} &= \norm{(-\vec{x} \oplus_E -\vec{u})\oplus_E\gyr{- \vec{x}, - \vec{u}}{(\vec{u}\oplus_E \vec{y})}}\\
{} &= \norm{-\vec{x}\oplus_E \vec{y}}.
\end{align*}
It follows that 
$$
d_E(L_\vec{u}(\vec{x}), L_\vec{u}(\vec{y})) = \tanh^{-1}\norm{-L_{\vec{u}}(\vec{x})\oplus_E L_{\vec{u}}(\vec{y})} = \tanh^{-1}\norm{-\vec{x}\oplus_E \vec{y}} = d_E(\vec{x}, \vec{y}). 
$$
This proves that $L_{\vec{u}}$ is an isometry of $(\B, d_E)$.
\end{proof}

\begin{corollary}\label{cor: gyroautomorphism as isometry}
The gyroautomorphisms of the Einstein gyrogroup are isometries of $\B$ with respect to $d_E$.
\end{corollary}
\begin{proof}
Let $\vec{u}, \vec{v}\in\B$. According to Theorem 10 (3) of \cite{TSKW2015ITG}, we have $$L_{\vec{u}}\circ L_{\vec{v}} = L_{\vec{u}\oplus_E\vec{v}}\circ\gyr{\vec{u}, \vec{v}}{}.$$ Hence, $\gyr{\vec{u}, \vec{v}}{} = L_{\vec{u}\oplus_E\vec{v}}^{-1}\circ L_{\vec{u}}\circ L_{\vec{v}} = L_{-(\vec{u}\oplus_E\vec{v})}\circ L_{\vec{u}}\circ L_{\vec{v}}$. This implies that $\gyr{\vec{u}, \vec{v}}{}$ is an isometry of $(\B, d_E)$, being the composite of isometries.
\end{proof}

In fact, Corollary \ref{cor: gyroautomorphism as isometry} is a special case of the following theorem.

\begin{theorem}\label{thm: automorphism as isometry}
Every automorphism of  $(\B, \oplus_E)$ that preserves the Euclidean norm is an isometry of $\B$ with respect to $d_E$. Therefore, every transformation in $\Or{\B}$ is an isometry of $\B$.
\end{theorem}
\begin{proof}
Let $\tau\in\aut{\B, \oplus_E}$ and suppose that $\tau$ preserves the Euclidean norm. Then $\tau$ is bijective. Direct computation shows that
$$
d_E(\tau(\vec{x}), \tau(\vec{y})) = \tanh^{-1}\norm{\tau(-\vec{x}\oplus_E \vec{y})} = \tanh^{-1}\norm{-\vec{x}\oplus_E \vec{y}} = d_E(\vec{x}, \vec{y})
$$
for all $\vec{x}, \vec{y}\in\B$. Hence, $\tau$ is an isometry of $(\B, d_E)$. The remaining part of the theorem is immediate since $\Or{\B}\subseteq \aut{\B, d_E}$.
\end{proof}

Next, we give a complete description of the isometry group of $(\B, d_E)$ using Abe's result \cite{TA2014GPM}.

\begin{theorem}\label{thm: isometry group of Einstein gyrogroup}
The isometry group of $(\B, d_E)$ is given by
\begin{equation}
\Iso{\B, d_E} = \cset{L_{\vec{u}}\circ \tau}{\vec{u}\in\B\textrm{ and }\tau\in\Or{\B}}.
\end{equation}
\end{theorem}
\begin{proof}
By Theorems \ref{thm: left gyrotranslation isometry} and \ref{thm: automorphism as isometry}, $$\cset{L_{\vec{u}}\circ \tau}{\vec{u}\in\B\textrm{ and }\tau\in\Or{\B}}\subseteq \Iso{\B, d_E}.$$ Let $\psi\in \Iso{\B, d_E}$. By definition, $\psi$ is a bijection from $\B$ to itself. By Theorem 11 of \cite{TSKW2015ITG}, $\psi = L_{\psi(\vec{0})}\circ\rho$, where $\rho$ is a bijection from $\B$ to itself that leaves $\vec{0}$ fixed. As in the proof of Theorem 18 (2) of \cite{TS2016TAG}, $ L_{\psi(\vec{0})}^{-1} = L_{-\psi(\vec{0})}$ and so $\rho = L_{-\psi(\vec{0})}\circ \psi$. Therefore, $\rho$ is an isometry of $(\B, d_E)$. Since $d_E(\rho(\vec{x}), \rho(\vec{y})) = d_E(\vec{x}, \vec{y})$ and $\tanh^{-1}$ is injective, it follows that
$$
\norm{-\rho(\vec{x})\oplus_E \rho(\vec{y})} = \norm{-\vec{x}\oplus_E\vec{y}}
$$
for all $\vec{x},\vec{y}\in\B$. Hence, $\rho$ preserves the Einstein gyrometric. By Theorem 3.1 of \cite{TA2014GPM}, $\rho = \res{\tau}{\B}$, where $\tau$ is an orthogonal transformation of $\R^n$. This proves the reverse inclusion.
\end{proof}

By Theorem \ref{thm: isometry group of Einstein gyrogroup}, every isometry of $(\B, d_E)$ has a (unique) expression as the composite of a left gyrotranslation and the restriction of an orthogonal transformation of $\R^n$ to the unit ball. According to the commutation relation (55) of \cite{TS2016TAG} for the case of the Einstein gyrogroup, one has the following composition law of isometries of $(\B, d_E)$:
\begin{equation}\label{eqn: composition law, map form}
(L_\vec{u}\circ \alpha)\circ(L_{\vec{v}}\circ \beta) = L_{\vec{u}\,\oplus_E\, \alpha(\vec{v})}\circ (\gyr{\vec{u}, \alpha(\vec{v})}{}\circ\alpha\circ\beta)
\end{equation}
for all $\vec{u},\vec{v}\in\B, \alpha,\beta\in\Or{\B}$. This reminds us of the composition law of Euclidean isometries. Note that $L_\vec{u}\circ \alpha = L_{\vec{v}}\circ \beta$, where $\vec{u},\vec{v}\in\B$ and $\alpha,\beta\in\Or{\B}$, if and only if $\vec{u} = \vec{v}$ and $\alpha = \beta$. This combined with \eqref{eqn: composition law, map form} implies that the map $L_{\vec{v}}\circ \tau \mapsto (\vec{v}, \tau)$ defines an isomorphism from the isometry group of $(\B, d_E)$ to the {\it gyrosemidirect product} $\B\rtimes_{\rm gyr}\Or{\B}$, which is a group consisting of the underlying set $$\cset{(\vec{v}, \tau)}{\vec{v}\in\B\textrm{ and }\tau\in\Or{\B}}$$ and group multiplication
\begin{equation}\label{eqn: group law of Bxgyr O(B)}
(\vec{u}, \alpha)(\vec{v}, \beta) = (\vec{u}\oplus_E \alpha(\vec{v}), \gyr{\vec{u}, \alpha(\vec{v})}{}\circ\alpha\circ\beta).
\end{equation}
For the relevant definition of a gyrosemidirect product, see Section 2.6 of \cite{AU2008AHG}.  Equation \eqref{eqn: group law of Bxgyr O(B)} is an analogous result in Euclidean geometry that the isometry group of $n$-dimensional Euclidean space $\R^n$ can be realized as the semidirect pro-duct $\R^n\rtimes \Or{\R^n}$. The result that the group of {\it holomorphic} automorphisms of a bounded symmetric domain can be realized as a gyrosemidirect product is proved by Friedman and Ungar in Theorem 3.2 of \cite{MR1290679}. Further, a characterization of continuous endomorphisms of the {\it three-dimensional} Einstein gyrogroup is obtained; see Theorem 1 of \cite{MR3455332}.

As an application of Theorem \ref{thm: isometry group of Einstein gyrogroup}, we show that the space $(\B, d_E)$ is homo-geneous; that is, there is an isometry of $(\B, d_E)$ that sends $\vec{u}$ to $\vec{v}$ for all arbitrary points $\vec{u}$ and $\vec{v}$ in $\B$. We also give an easy way to construct {\it point-reflection} symmetries of the unit ball.

\begin{theorem}[Homogeneity]
If $\vec{u}$ and $\vec{v}$ are arbitrary points in $\B$, then there is an isometry $\psi$ of $(\B, d_E)$ such that $\psi(\vec{u}) = \vec{v}$. In other words, $(\B, d_E)$ is homogeneous.
\end{theorem}
\begin{proof}
Let $\vec{u}, \vec{v}\in\B$. Define $\psi = L_{\vec{v}}\circ L_{-\vec{u}}$. Then $\psi$ is an isometry of $(\B, d_E)$, being the composite of isometries. It is clear that $\psi(\vec{u}) = \vec{v}\oplus_E(-\vec{u}\oplus_E \vec{u}) = \vec{v}$.
\end{proof}

\begin{theorem}[Symmetry]
For each point $\vec{v}\in\B$, there is a point-reflection $\sigma_{\vec{v}}$ of $(\B, d_E)$ corresponding to $\vec{v}$; that is, $\sigma_{\vec{v}}$ is an isometry of $(\B, d_E)$ such that $\sigma_{\vec{v}}^2$ is the identity transformation of $\B$ and $\vec{v}$ is the unique fixed point of $\sigma_{\vec{v}}$.
\end{theorem}
\begin{proof}
Let $\iota$ be the negative map of $\B$; that is, $\iota(\vec{w}) = -\vec{w}$ for all $\vec{w}\in\B$. In view of \eqref{eqn: Euclidean Einstein addition}, it is clear that $\iota$ is an automorphism of $\B$ with respect to $\oplus_E$. By Theorem \ref{thm: automorphism as isometry}, $\iota$ is an isometry of $(\B, d_E)$. Define $\sigma_\vec{v} = L_{\vec{v}}\circ \iota\circ L_{-\vec{v}}$. Then $\sigma_{\vec{v}}$ is an isometry of $(\B, d_E)$ that is a point-reflection of $\B$ corresponding to $\vec{v}$. The uniqueness of the fixed point of $\sigma_{\vec{v}}$ follows from the fact that $\vec{0}$ is the unique fixed point of $\iota$.
\end{proof}

\vspace{0.3cm}
\noindent{\bf Acknowledgements.} The author would like to thank Themistocles M. Rassias for his generous collaboration. He also thanks anonymous referees for useful \mbox{comments}.

\bibliographystyle{amsplain}\addcontentsline{toc}{section}{References}
\bibliography{References}
\end{document}